\documentclass[a4paper,twoside,11pt,notitlepage]{amsart}
\usepackage{a4wide}
\usepackage[T1]{fontenc}
\usepackage{amsthm}
\usepackage[latin1]{inputenc}
\usepackage{amsfonts,amsmath,epsfig,amscd,amssymb,latexsym}

\theoremstyle{plain}
\newtheorem{theorem}{Theorem}
\newtheorem{lemma}[theorem]{Lemma}

\newtheorem{corollary}[theorem]{Corollary}

\theoremstyle{definition}
\newtheorem{definition}[theorem]{Definition}

\newtheorem{remark}[theorem]{Remark}

\title{On an approximate functional equation involving the divisor function}
\author{Anne-Maria Ernvall-Hyt\"onen} \address{Department of Mathematics and Statistics\\
         University of Helsinki\\
         Finland}
\thanks{Currently the author is working also for the University of Oulu. Approximately half of the work was done 
while working at the KTH in Stockholm on a grant no. 2009-721 from Vetenskapsr{\aa}det, and the rest of the work supported by an Academy of Finland grants no.138337 and 138522. Finally, the author would like to thank the anonymous referee for thorough comments.}
\newcommand{\ud}{\textup{d}}

\begin{document}\begin{abstract}
The error term in the approximate functional equation for exponential sums 
involving the divisor function will be improved under certain conditions 
for the parameters of the approximate functional equation. 
\end{abstract}\maketitle \section{Introduction} Exponential sums
\[
D(M_1,M_2;\alpha)=\sum_{M_1\leq n\leq M_2}d(n)e(\alpha n)
\] with $\alpha \in \mathbb{R}$, $e(x)=e^{2\pi i x}$, involving the divisor function $d(n)=\sum_{d|n, d>0}1$ are related to the exponential sums of Fourier coefficients of cusp forms. The main difference between the situations is that while the Fourier coefficients change signs, the divisor function is always positive. Therefore, one cannot obtain good general upper bounds. For instance
\[
D(1,M;0)=\sum_{1\leq n\leq M}d(n)\asymp M\log M,
\]
while the average estimate is much smaller:
\[
\int_0^1\left|\sum_{1\leq n\leq M}d(n)e(n\alpha)\right|^2\ud \alpha =\sum_{1\leq n\leq M}d(n)^2\asymp M\log^3 M,
\]
and therefore, in average the estimate is $\sqrt{M}\log^{3/2} M$. The situation is different for the Fourier coefficients of holomorphic cusp forms: the upper bound $\sum_{n\leq M}a(n)e(n\alpha)\ll M^{1/2}$ \cite{jutila:ramanujan}, is the same as the average estimate.

The approximate functional equation involving the divisor function has been studied by Jutila \cite{jutila:divisor}, and even earlier by Wilton \cite{wilton:divisor}. Jutila's result reads as follows:
\begin{theorem} Let $M\geq 2$, $1\leq k\leq M^{1/2}$ be an integer,  and $0<|\eta|\leq k^{-2}$. Further assume $h$ is a positive integer with gcd$(h,k)=1$ and with the property $h\overline{h}\equiv 1\bmod k$. If $k^2\eta^2M\gg 1$, then for $M\leq M_1<M_2\leq 2M$ we have
\[
D\left(M_1,M_2;\frac{h}{k}+\eta\right)=(k|\eta|)^{-1}D\left(k^2\eta^2M_1,k^2\eta^2M_2;-\frac{\bar{h}}{k}-(k^2\eta)^{-1}\right)+O(M^{1/2}\log M).
\]
\end{theorem}
This is sharp because, for instance, when $M$ is a square, choosing $k=\sqrt{M}$, $h=1$ and $\eta=\frac{1}{M}$, we have
\[
D\left(M,M+\frac{1}{2}\sqrt{M};\frac{1}{k}+\eta\right)\asymp \sqrt{M}\log M,
\]
and
\[
(k\eta)^{-1}D\left(\eta^2k^2M,k^2\eta^2\left(M+\frac{1}{2}\sqrt{M}\right),-\frac{1}{\sqrt{M}}-1\right)=d(1)e\left(-\frac{1}{\sqrt{M}}-1\right)\sqrt{M}\ll \sqrt{M}.
\]
However, under certain conditions, this result can be improved. The definition of a Farey sequence and a Farey approximation are needed to formulate the theorem. They can be found in Definition \ref{farey}.
\begin{theorem}\label{main}
There exists an absolute const ant $A\geq 1$ such that for any $\varepsilon,\varepsilon'>0$ there exists some $a>0$ such that the following holds: Let $c$ be an arbitrary positive real number; Let $M,k,h$ be arbitrary integers satisfying $M\geq 2$, $1\leq k\leq M^{1/4}$ and $(h,k)=1$, and let $\eta$ be a real number satisfying $|\eta|\leq M^{-1/4}k^{-1}$ and $k^2\eta^2M\geq c(\log M)^A$. For each $j\in \mathbb{Z}^+$ set $\Delta_j=k^2|\eta|^{3/2}M^{1/2}(k^2\eta^2M)^{\varepsilon}2^{-j}$, and let $\ell$ be the smallest positive integer for which $\Delta_{\ell}\leq c^{-1}(k^2\eta^3M)^{2/5}$. For each $j\in \{1,2,\dots ,\ell\}$ we assume that the number $\beta:=-\frac{\bar{h}}{k}-(k^2\eta)^{-1}$ has a Farey approximation $\frac{h_j}{k_j}$ of order $\Delta_j^{1/2-\varepsilon'}$ which satisfies $\left|\beta-\frac{h_j}{k_j}\right|\geq c\Delta_j^{\varepsilon-1}$ or $k_j\geq c\Delta_j^{5/6}(k^2\eta^2M)^{-1/3}$. Then for any integers $M_1,M_2$ with $M\leq M_1<M_2\leq 2M$ we have
\[
D\left(M_1,M_2;\frac{h}{k}+\eta\right)=(k|\eta|)^{-1}D\left(k^2\eta^2M_1,k^2\eta^2M_2;-\frac{\bar{h}}{k}-(k^2\eta)^{-1}\right)+O\left(M^{1/2}(k^2\eta^2M)^{-a}\right),
\]
for some positive constant $a$, where the implied constant depends only on $\varepsilon,\varepsilon',c$.
\end{theorem}
\begin{remark}
The conditions for $\beta$ seem technical and possibly fairly restricting. However, it will be noted that the set of $\beta$ with $0\leq \beta\leq 1$ that do not satisfy either of these conditions is of Lebesgue measure at most $O(M^{-1/4+\varepsilon})$.
\end{remark}

This article is strongly connected to \cite{ame:approximate}. These papers also have some overlapping statements. This was impossible to avoid without sacrificing readability. Generally, this article connects techniques from that paper, and from \cite{jutila:ramanujan}.

\section{Notation, preliminaries and known lemmas}

From now on, let $U=\sqrt{M}|\eta|^{-1/2}\left(k^2\eta^2M\right)^d$ with $d$ a very small fixed positive number. Furthermore, we will assume that $k\ll M^{-1/4}$ and $k^2\eta^2M\gg 1$, since these are assumptions needed to prove the main result.

To shorten the notation, write $M_{-1}=M-JU$ and $M_2=M+\Delta+JU$ with $J$ a constant. Since $k^2\eta^2M\gg 1$, we have $|\eta|^{-1/2}\ll M^{1/4}k^{1/2}$, and hence
\[
U=M^{1/2}|\eta|^{-1/2}(k^2\eta^2M)^d\ll M^{1/2+1/4}k^{1/2}(k^2\eta^2M)^d\ll M^{7/8}(k^2\eta^2M)^d=o(M),
\]
since $d$ is very small.

Throughout the paper, we assume the assumptions of Theorem \ref{main}. Furthermore, a \emph{smooth weight function} $w$ defined on some interval, say, on interval $[M,M+\Delta]$ is a function satisfying the following conditions:
\begin{enumerate}
\item The support of $w$ lies on the interval $[M,M+\Delta]$.
\item The function $w$ satisfies the condition $w^{(j)}\ll \Delta^{-j}$ for $0\leq j\leq P$ for some $P$ which typically depends on the $\varepsilon$ needed.
\end{enumerate}
In practice, we will also require $w(x)=1$ for all $x$ on some suitable subinterval of $[M,M+\Delta]$.

Farey sequences and Farey approximations are needed throughout the article. They are defined in the following way:

\begin{definition}\label{farey} A \emph{Farey sequence} of order $Q$ consists of all rational numbers $\frac{a}{b}$ on some interval $I$ (in this paper $I=[0,1]$) with the property gcd$(a,b)=1$ and $1\leq b\leq Q$. The \emph{Farey approximation} of order $Q$ is the approximation of a number as a sum $\frac{a}{b}+\eta$, where $\frac{a}{b}$ belongs to a Farey sequence of order $Q$ and $|\eta|\leq \frac{1}{bQ}$.
\end{definition}

Write $\tilde{D}(M_1,M_2,\alpha)$ for the smoothed sum:
\[
\tilde{D}(M_1,M_2,\alpha)=\sum_{M_1\leq n\leq M_2}d(n)e(\alpha n)w(n),
\]
where $w$ is a suitable smooth weight function. Sometimes it is easier to take a main term out of the sum, and consider the rest:
\[
\tilde{\Delta}(M_1,M_2,\alpha)=\tilde{D}(M_1,M_2,\alpha)-k^{-1}\int_{M_1}^{M_2}\left(\log x +2\gamma-2\log k\right) e\left(\eta x\right)w(x) \ud x.
\]
This remaining part behaves a lot like exponential sums of Fourier coefficients of holomorphic cusp forms. Actually, using the word "main term" is a bit misleading here as it often happens that the remaining term is actually larger than the main term. For instance, let us consider the following expression, where $w$ is compactly supported on the interval $[M,M+\Delta]$, an satisfies the condition $w^{(j)}\ll \Delta$ for $0\leq j\leq P$, where $P$ is a sufficiently large positive number, and attains the value $1$ on some subinterval of length $\asymp \Delta$:
\[
\tilde{\Delta}(M,M+\frac{M^{3/4}}{4},M^{-1/2})\asymp \tilde{D}\left(M,M+\frac{M^{3/4}}{4},M^{-1/2}\right)\asymp M^{1/2}
\]
(see \cite{e&k}) but
\[
k^{-1}\int_{M}^{M+\Delta}\left(\log x +2\gamma-2\log k\right) e\left(\eta x\right)w(x) \ud x\ll 1.
\]
However, for instance
\[
\sum_{n\leq M}d(n)\asymp (\log M+2\gamma-1)M,
\]
and hence, in some very important cases this main term is also the actual, not just a symbolic, main term.

The Voronoi type transformation formula (see e.g. Theorem 1.7 \cite{jutila:lectures}) gives a nice expression for the term $\tilde{\Delta}(M,M+\Delta;\alpha)$:
\begin{multline}\label{voronoi}
\tilde{D}(M,M+\Delta;\alpha)=k^{-1}\int_M^{M+\Delta}(\log x+2\gamma-2\log k)w(x)e(\eta x)\ud x\\+k^{-1}\sum_{n=1}^{\infty}d(n)\int_M^{M+\Delta}\left(-2\pi e\left(-n\frac{\bar{h}}{k}\right)Y_0\left(4\pi\frac{\sqrt{nx}}{k}\right)+4\left(n\frac{\bar{h}}{k}\right)K_0\left(4\pi\frac{\sqrt{nx}}{k}\right)\right)w(x)e(\eta x)\ud x
\end{multline}

For treating oscillating integrals the following lemma by Jutila and Motohashi (\cite{jutimoto:acta}, Lemma 6) is very helpful:

\begin{lemma}\label{jutilamotohashi}
Let $A$ be a $P\geq 0$ times differentiable function which is
compactly supported in a finite interval $[a,b]$. Assume also that
there exist two quantities $A_0$ and $A_1$ such that for any
non-negative integer $\nu\leq P$ and for any $x\in [a,b]$,
\[
A^{(\nu)}(x)\ll A_0A_1^{-\nu}.
\]
Moreover, let $B$ be a function which is real-valued on $[a,b]$,
and regular throughout the complex domain composed of all points
within the distance $\varrho$ from the interval; and assume that
there exists a quantity $B_1$ such that
\[
0<B_1\ll \left|B'(x)\right|
\]
for any point $x$ in the domain. Then we have
\[
\int_{a}^{b}A(x)e\left(B(x)\right)\ud x \ll
A_0\left(A_1B_1\right)^{-P}\left(1+\frac{A_1}{\varrho}\right)^P\left(b-a\right).
\]
\end{lemma}

For treating integrals which do not oscillate sufficiently fast, we need the saddlepoint lemma (Lemma 3 in \cite{jutila:ramanujan}).
Before that we need some definitions.
\begin{definition}
Define $\eta_{J}(x)$ such that it satisfies the following equation
for any integrable funtion $h$:
\begin{equation}\label{erikoispaino}
V^{-J}\int_0^{V}\ud u_1 \cdots \int_0^{V}\ud
u_{J}\int_{M_{-1}+u}^{M_2-u}h(x)\ud
x=\int_{M_{-1}}^{M_2}\eta_{J}(x)h(x)\ud x
\end{equation}
where $u=u_1+\cdots +u_{J}$ and $V<(M_2-M_{-1})/2J$, and define
$\eta_0$ to be the characteristic function of the interval
$[M_{-1},M_2]$.
\end{definition}
One may easily compute the Fourier transform
\[
\hat{\eta}_J(\lambda)=-\left(i\lambda\right)^{J+1}\left(e^{-i\lambda
U}-1\right)^{J}\left(e^{-i\lambda M_2}-e^{-i\lambda
M_{-1}}\right).
\]
This implies that $\eta_J(x)$ is $J-1$ times differentiable.\\

Denote by $D$ the complex domain consisting of points $z$
satisfying the condition $|z-x|<\mu$ for some $x\in [M_{-1},M_2]$,
where $\mu \asymp M_{-1}$.

\begin{lemma}[Saddle-point lemma]\label{satulapistelemma}
Let $F\gg M^{\varepsilon}_{-1}$, $G$ and $V$ be positive constants. Let $f$ be a
real function such that
\[
f'(x)\ll \frac{F}{M_{-1}},\ f''(x)>0,\ \left|f''(x)\right|\asymp
FM_{-1}^2\ \textrm{for}\ x\in[M_{-1},M_2].
\]
Let $g$ be a holomorphic function with $g(z)\ll G$ in the domain
$D$. Denote the characteristic function of
\begin{equation}\label{joukko}
(M_{-1},M_{-1}+JV)\bigcup(M_2-JV,M_2)
\end{equation}
by $\delta(x)$. Let $x_0$ be the (possibly existing) zero of
$f'(x)+\iota$ in the interval $(M_{-1},M_2)$, and suppose that $V\gg
\delta(x_0)F^{-1/2}M_{-1}$. Write
\[
E_J(x)=\frac{G}{\left(\left|f'(x)+\iota\right|+F^{1/2}/M_{-1}\right)^{J+1}}.
\]
Then
\begin{multline*}
\int_{M_{-1}}^{M_2}\eta_{J}g(x)e\left(f(x)+\iota x\right)\ud
x=\xi_{J}(x_0)g(x_0)f''(x_0)^{-1/2}e\left(f(x_0)+\iota
x_0+1/8\right)
\\+
O\left(\left(1+\delta(x_0)F^{1/2}\right)GM_{-1}F^{-3/2}\right)+O\left(V^{-J}\sum_{j=0}^{J}\left(E_{J}\left(M_{-1}+jV\right)+E_{J}\left(M_{2}-jV\right)\right)\right),
\end{multline*}
where $\xi_J$ is a bounded function on $(M_{-1},M_2)$ with the
following properties:
\begin{itemize}
\item $\xi_J(x)=1$ on $(M_{-1}+JV,M_2-JV)$
\item $\xi_J'(x)$ is continuous and $\xi_J'(x)\ll V^{-J}$ on the
set (\ref{joukko}), except possibly at the points $M_{-1}+jV$,
$M_2-jV$, $j=1,\dots,J-1$.
\end{itemize}
If $x_0$ does not exist, then the terms and conditions involving
$x_0$ are to be omitted.
\end{lemma}

\section{Technical lemmas}

\begin{lemma}\label{paa} Let $|\eta| \gg \frac{1}{\Delta^{1-\varepsilon}}$, and $w(x)$ be a smooth weight function satisfying the condition $w^{(j)}\ll \Delta^{-j}$ for $j\leq P$ for a sufficiently large $P$. Then
\[
k^{-1}\int_{M}^{M+\Delta}\left(\log x +2\gamma-2\log k\right) e\left(\eta x\right)w(x) \ud x\ll k^{-1}\log M.
\]
\end{lemma}
\begin{proof}
Let us use partial integration $P$ times to obtain
\[
\int_{M}^{M+\Delta}\left(\log x +2\gamma-2\log k\right) e\left(\eta x\right)w(x) \ud x\ll \log M \Delta^{1-P}\eta^{-P}\ll 1.
\]
\end{proof}

The following lemma will be formulated in the general case, but in practice it will be used when $|\eta|$ is too small for the use of the previous lemma.

\begin{lemma}\label{paa2}
Let $\Delta\ll k^{6/5}M^{2/5}$. Then
\[
k^{-1}\int_{M}^{M+\Delta}\left(\log x +2\gamma-2\log k\right) e\left(\eta x\right) \ud x\ll \Delta^{1/6}M^{1/3+\varepsilon}
\]
\end{lemma}
\begin{proof} The proof is very simple. Just estimate using absolute values, and use the assumption for $\Delta$.
\end{proof}

\begin{lemma}\label{K-termi} We have
\[
k^{-1}\sum_{n=1}^{\infty}d(n)e\left(\frac{n\bar{h}}{k}\right)\int_M^{M+\Delta}K_0\left(4\pi\frac{\sqrt{nx}}{k}\right)e\left(\eta x\right)\ud x\ll M^{-3/8}.
\]
\end{lemma}
\begin{proof}
Let us first use the asymptotic bound for the $K$-Bessel function (see e.g. \cite{lebedev} (5.11.9)):
\[
K_0(z)\ll \left(\frac{\pi}{2z}\right)^{1/2}e^{-z}.
\]
We have
\begin{multline*}
k^{-1}\sum_{n=1}^{\infty}d(n)e\left(\frac{n\bar{h}}{k}\right)\int_M^{M+\Delta}K_0\left(4\pi\frac{\sqrt{nx}}{k}\right)e\left(\eta x\right)\ud x\\ \ll 
\left|k^{-1}\sum_{n=1}^{\infty}d(n)e\left(\frac{n\bar{h}}{k}\right)\int_M^{M+\Delta}\left(\frac{k}{8\sqrt{nx}}\right)^{1/2}e^{-4\pi\frac{\sqrt{nx}}{k}}e\left(\eta x\right)\ud x\right| \\ \ll k^{-1/2}\Delta M^{-1/4}\sum_{n=1}^{\infty}n^{\varepsilon-1/4}e^{-4\pi\frac{\sqrt{nM}}{k}} \ll k^{-1/2}\Delta M^{-1/4}\sum_{n=1}^{\infty}\left(\frac{k}{\sqrt{nM}}\right)^4n^{-1/4}\\ \ll k^{7/2}\Delta M^{-1/4-2}\ll M^{-1/4-2+1+7/8}\ll M^{-3/8}.
\end{multline*}
This proves the lemma
\end{proof}\section{Estimates for short sums}
We are interested in getting good estimates for short linear sums, but in order to get those, we need estimates for short non-linear sums. We will start by quoting a result which is probably not very well known.
\begin{theorem}
Let $\eta,B\in \mathbb{R}$, and denote $F=|B|M^{1/2}$. We assume that  $M^2\ll \Delta F$. Let
\[
f(z)=\eta z+Bz^{1/2}.
\]
Let $g\in C^{1}[M,M+\Delta]$ and
\[
||g||_{\infty}\leq G, ||g'||_{\infty}\leq G'.
\]
Then
\[
\sum_{M\leq n\leq M+\Delta}d(m)g(m)e(f(m))\ll \left(\frac{\Delta}{M}\right)^{5/6}(G+\Delta G')M^{1/2}F^{1/3+\varepsilon}.
\]\end{theorem}
\begin{proof}
The full proof can be found in Karppinen's licentiate thesis \cite{karppi}. However, as that particular work is fairly difficult to come by, and available only in Finnish, I briefly tell what modifications have to be done to the proof of Theorem 4.1 in \cite{e&k}.

For simplicity, assume that $g(x)=1$.  Use Theorem 3.3 of \cite{jutila:lectures}. Most of the terms can be treated similarly. However, with divisor function, we have to deal with a main term:
\[
F^{-1/2}M\log F \sum_{r\in I }\frac{1}{k}\ll \Delta^{2/3}M^{-1/6}F^{1/6+\varepsilon},
\]
where $I$ corresponds the same Farey partitioning as in the original proof. The second difference is the error term arising from a change of a condition in the original theorem (cusp forms versus the divisor function). This gives an error of size $O\left(\Delta^{2/3}M^{1/2}F^{-5/6+\varepsilon}\right)$
\end{proof}

Finally, we may formulate the estimate for short linear sums. This theorem is similar to Theorem 5.1 \cite{e&k}, and there are very few differences in proofs.

\begin{theorem} Let $1\leq \Delta \leq M$, let the Farey approximation of order $\Delta^{1
/2-\varepsilon_2}$ of a parameter $\alpha$ be $\alpha=\frac{h}{k}+\eta$, where $|\eta|\gg \frac{1}{\Delta^{1-\varepsilon}}$ or $\Delta\ll k^{6/5}M^{2/5}$, and $\varepsilon_2$ is a sufficiently small fixed positive real number depending only on $\varepsilon$. Let $w(x)$ be s smooth weight function defined on the interval $[M,M+\Delta]$ and satisfying the condition $w^{(j)}(x)\ll \Delta^{-j}$ for $0\leq j\leq P$, where $P$ is a large positive integer depending on $\varepsilon$.
\begin{enumerate}
\item If $|\eta|\leq \Delta^{-1+\varepsilon_2}$, then
\[
\sum_{M\leq n\leq M+\Delta}d(n)e(n\alpha)w(n)\ll M^{1/3+\varepsilon}\Delta^{1/6}
\]
\item Otherwise, further assuming
\[
\Delta\ll\frac{\sqrt{M}}{\sqrt{|\eta|}},
\]
we have
\[
\sum_{M\leq n\leq M+\Delta}d(n)e(n\alpha)w(n)\ll M^{1/3+\varepsilon}\Delta^{1/6}+k^{-1}\Delta\left|\eta\right|^{-1/2}M^{-1/2}\left(k^2\eta^2M\right)^{\varepsilon}.
\]
\end{enumerate}
\end{theorem}

\begin{proof}
Proof is similar to \cite{e&k}, but requires the use of Lemmas \ref{paa} and \ref{paa2} for treating the main term. Furthermore, the current second term $k^{-1}\Delta\left|\eta\right|^{-1/2}M^{-1/2}\left(k^2\eta^2M\right)^{\varepsilon}$ is slightly better than the original second term $k^{-1}\Delta\left|\eta\right|^{-1/2}M^{-1/2}M^{\varepsilon}$. This improvement is already present in the formulation of the theorem in \cite{ame:approximate}, and this comes from the fact, that this term corresponds to the case when there is one term in the Voronoi type summation formula that does not oscillate much, and in this case this term is just estimated using absolute values. Now, the Fourier coefficient of this term, say $a(n_0)$ corresponds to some value of $n_0$ with $n_0\asymp k^2\eta^2M$, and therefore, we have $a(n_0)\ll (k^2\eta^2M)^{\varepsilon}$ instead of the originally used weaker bound $a(n_0)\ll M^{\varepsilon}$.
\end{proof}
\begin{corollary}\label{lyhyet}
With the assumptions of the previous theorem, and the additional assumption $\Delta\ll M^{5/8}$, we have
\[
\sum_{M\leq n\leq M+\Delta}d(n)e(n\alpha)w(n)\ll M^{1/3+\varepsilon}\Delta^{1/6}.
\]
\end{corollary}
\begin{proof} If $|\eta|\ll \Delta^{\varepsilon-1}$, then this is the special case of the previous theorem. If $|\eta|\gg \Delta^{-1+\varepsilon}$, then also $|\eta|\gg \Delta^{-1}$, and the result follows by a simple manipulation:
\[
k^{-1}\Delta\left|\eta\right|^{-1/2}M^{-1/2}\left(k^2\eta^2M\right)^{\varepsilon}\ll \Delta\left|\eta\right|^{-1/2}M^{-1/2+\varepsilon}\ll \Delta^{3/2}M^{-1/2+\varepsilon}\ll \Delta^{1/6}M^{1/3+\varepsilon}.
\]
\end{proof}

The assumptions of the theorem and the corollary may look fairly restricting, but actually, a very large proportion of all $\alpha\in [0,1]$ satisfy the conditions when $\Delta=o\left(M^{1-\delta_1}\right)$ for any fixed $\delta_1>0$: The measure of the set of $\alpha$ which do no satisfy the conditions can be easily estimated. If $\alpha=\frac{h}{k}+\eta$ does not satisfy the conditions, then
\[
|\eta|\ll \frac{1}{\Delta^{1-\varepsilon}} \ \textrm{and}\ k\ll \Delta^{5/6} M^{-1/3}.
\]
Now the measure of this set is for any given $k$ at most
\[
\ll \frac{k}{\Delta^{1-\varepsilon}},
\]
and now summing over the values of $k$:
\[
\sum_{1\leq k\ll \Delta^{5/6} M^{-1/3}}\frac{k}{\Delta^{1-\varepsilon}}\ll (\Delta^{5/6} M^{-1/3})^2\Delta^{\varepsilon-1}\ll M^{-2/3}\Delta^{\varepsilon+2/3},
\]
and hence, the measure of the set of $\alpha\in [0,1]$ not satisfying the conditions is at most $\Delta^{\varepsilon+2/3}M^{-2/3}$, which is small when $\Delta=o\left(M^{1-\delta_1}\right)$ for any fixed positive $\delta_1$, because the $\varepsilon$ can be required to be smaller than this $\delta_1$.

This corollary and Lemmas \ref{paa} and \ref{paa2} together give an estimate for short smoothed sums (notice that even though Lemma \ref{paa2} is formulated for non-smoothed sums, this does not affect anything, because in the proof of the lemma, we only used absolute values). However, under certain conditions the weight function can be removed:
\begin{theorem}\label{painotpois}
Let $\Delta\ll M^{5/8}$ and let $\alpha$ have Farey approximations $\frac{h_j}{k_j}+\eta_j$ of orders
\[
\Delta^{1/2-\varepsilon}, \frac{1}{2}\Delta^{1/2-\varepsilon}, \left(\frac{1}{2}\right)^2\Delta^{1/2-\varepsilon}, \dots, M^{2/5-\delta}
\]
with $|\eta_j|\gg \Delta_j^{1-\varepsilon}$ or $\Delta_j\ll k_j^{6/5}M^{2/5}$ where $\Delta_j$ stands for $\left(\frac{1}{2}\right)^j\Delta$. Then
\[
\sum_{M\leq n\leq M+\Delta}d(n)e\left(\frac{h}{k}+\eta\right)\ll \Delta^{1/6}M^{1/3+\varepsilon}.
\]
\end{theorem}

Conditions again look fairly technical and restricting. However, we may show that the measure of set of $\alpha\in[0,1]$ not satisfying these conditions is small: We proved earlier that for any $\Delta$, the measure of the set of $\alpha\in [0,1]$ which do not satisfy the conditions, is $O\left(\Delta^{2/3+\varepsilon}M^{-2/3}\right)$. Using this estimate and dyadic intervals, we obtain that also here the total measure of the set of $\alpha$ not satisfying the conditions is $O\left(\Delta^{2/3+\varepsilon}M^{-2/3}\right)$. Since $\Delta \ll M^{5/8}$, this yields $O\left(M^{-1/4+\varepsilon}\right)$ to be the total measure of the set for which the corollary does not hold.

\begin{proof} Proof is the similar to the proof of 5.5 \cite{e&k}. However, we sketch for sake of completeness. Also, some typos are fixed here. Let $\ell>0$. Write
\[\quad \left\{\begin{array}{l}\Delta_0=\frac{2\Delta}{5}\\ \Delta_1=\frac{5\Delta_0}{8}\\ \Delta_{\ell}=\frac{\Delta_{\ell-1}}{2}, \ \textrm{when $\ell\geq 2$}\end{array}\right.
 \quad \textrm{and}\quad \left\{\begin{array}{l}M_0=M+\frac{\Delta}{2}-\frac{\Delta}{5}\\ M_{\ell}=M_0+\frac{3\Delta_0}{4}+\frac{4}{5}\sum_{i=1}^{\ell-1}\Delta_i\\ \end{array}\right.
 \]
 Consider the set of weight functions $\{w_{\pm \ell}\ |\ 0\leq \ell\leq L\}$, where $L$ is large enough such that $\Delta_L\ll M^{2/5+\varepsilon}$ satisfying the following conditions:
\begin{itemize}
\item The support of the function is bounded:
\[
w_{0}(x)=\left\{\begin{array}{ll}1, & \textrm{when } x\in\left[M_0+\frac{\Delta_0}{4},M_0+\frac{3\Delta_0}{4}\right]\\ 0, & \textrm{when }x\leq M_0\\ & \textrm{or }x\geq M_0+\Delta_0\end{array}\right.
\]
and when $\ell>0$,
\[
w_{\ell}(x)=\left\{\begin{array}{ll}1, & \textrm{when } x\in\left[M_{\ell-1}+\Delta_{\ell-1},M_{\ell+1}\right]\\ 0, & \textrm{when }x\leq M_{\ell}\\ & \textrm{or }x\geq M_{\ell}+\Delta_{\ell}.\end{array}\right.
\]
\item The derivatives are assumed to satisfy the following condition
\[
w_{\ell}^{(j)}(x)\ll \left(\Delta_{\ell}\right)^{-j}
\]
for $0\leq j\leq J$ for some suitable value of $J$.
\item On the interval
\[
\left[M_{\ell+1},M_{\ell}+\Delta_{\ell}\right]
\]
the functions $w_{\ell}(x)$ and $w_{\ell+1}$ add to one:
\[
w_{\ell+1}(x)=1-w_{\ell}(x).
\]
\end{itemize}
Define $M_{-\ell}$, $\Delta_{-\ell}$ and $w_{-\ell}(x)$ symmetrically with respect to the line $x=M+\frac{\Delta}{2}$.

The situation will look like the following:\\
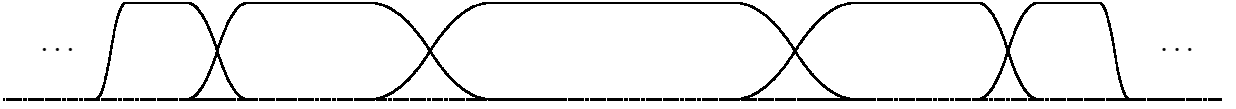
\begin{figure}[h]
\setlength{\unitlength}{0.8mm}
\begin{picture}(200,16)(-20,0)
\put(-14,8){$\dots$}
\put(170,8){$\dots$}
\qbezier(-20,0)(80,0)(180,0)%
\qbezier(60,16)(80,16)(100,16)%
\qbezier(60,16)(55,16)(50,8)%
\qbezier(50,8)(45,0)(40,0)%
\qbezier(100,16)(105,16)(110,8)%
\qbezier(110,8)(115,0)(120,0)%
\qbezier(40,16)(45,16)(50,8)%
\qbezier(50,8)(55,0)(60,0)%
\qbezier(120,16)(115,16)(110,8)%
\qbezier(110,8)(105,0)(100,0)%
\qbezier(20,16)(30,16)(40,16)%
\qbezier(120,16)(130,16)(140,16)%
\qbezier(20,16)(17.5,16)(15,8)%
\qbezier(15,8)(12.5,0)(10,0)%
\qbezier(140,16)(142.5,16)(145,8)%
\qbezier(145,8)(147.5,0)(150,0)%
\qbezier(10,16)(12.5,16)(15,8)%
\qbezier(15,8)(17.5,0)(20,0)%
\qbezier(150,16)(147.5,16)(145,8)%
\qbezier(145,8)(142.5,0)(140,0)%
\qbezier(0,16)(5,16)(10,16)%
\qbezier(150,16)(155,16)(160,16)%
\qbezier(0,16)(-1.25,16)(-2.5,8)%
\qbezier(-2.5,8)(-3.75,0)(-5,0)%
\qbezier(160,16)(161.25,16)(162.5,8)%
\qbezier(162.5,8)(163.75,0)(165,0)%

\end{picture}
\end{figure}

Use Corollary \ref{lyhyet} to obtain the bound $\Delta_i^{1/6}M^{1/3+\varepsilon}$ on the interval $[M_i,M_i+\Delta_i]$. Summing these bounds together gives the total bound $\Delta^{1/6}M^{1/3+\varepsilon}$, as desired. Finally, define now a new function
\[
W_0=\sum_{\ell}w_{\ell}.
\]
Notice that
\[
\left|\sum_{M\leq n\leq \Delta}d(n)e(n\alpha)-\sum W_0(n)d(n)e(n\alpha)\right|\ll M^{2/5+\varepsilon},
\]
which completes the proof of the theorem.

\end{proof}

\begin{lemma} When $k\leq M^{1/4}$, we have
\[
\Delta(M,h/k)\ll x^{3/8+\varepsilon}k^{1/4}+\min(x^{\varepsilon}k^2,x^{1/2-c_2})
\]
for some positive constant $c_2$.
\end{lemma}

\begin{proof} Proof is similar to the proof of \cite{jutila:ramanujan} Lemma 2.
\end{proof}

\begin{theorem}\label{painoton}
For $\alpha=\frac{h}{k}+\eta$ with $k\leq M^{1/4}$, $|\eta|\leq k^{-1}M^{-1/4}$ and $\left(Mk^2\eta^2\right)^b\gg \log M$ for some $0<b<\frac{1}{2}$ we have positive constants $a$ and $b$ such that
\[
D(x_1,x ;\alpha)\ll M^{1/2}(k^2\eta^2M)^{-a}
\]
for $|x-x_1|\ll |\eta|^{-1/2}M^{1/2+b}$ and $x,x_1\asymp M$
\end{theorem}
Notice that the condition $\left(Mk^2\eta^2\right)^b \gg \log M$ excludes those $\alpha$ for which $\eta\ll \left(\log M\right)^{1/2b}k^{-1}M^{-1/2}$, and the total size of this set, whenever $k\leq M^{1/4}$ is
\[
\ll \sum_{1\leq k\leq M^{1/4}}\sum_{1\leq h\leq k} \left(\log M\right)^{1/2b}k^{-1}M^{-1/2}=\sum_{1\leq k\leq M^{1/4}}\left(\log M\right)^{1/2b}M^{-1/2}\ll M^{-1/4+\varepsilon},
\] 
so the estimate holds for almost all $\alpha$. We are now ready to move to the proof.
\begin{proof} For simplicity. write
\[
D(t,\alpha)=\sum_{n\leq t}d(n)e(n\alpha).
\]
We have
\begin{equation}\label{alku}
D(x_1,x;\alpha)=e(\eta x)D\left(x;\frac{h}{k}\right)-e(\eta x_1)D\left(x_1;\frac{h}{k}\right)-2\pi i \eta\int_{x_1}^x D\left(t;\frac{h}{k}\right)e(\eta t)\ud t.
\end{equation}
Write now
\[
D\left(t;\frac{h}{k}\right)=k^{-1}(\log t+2\gamma-1-2\log k)t+E\left(0,\frac{h}{k}\right)+\Delta\left(t,\frac{h}{k}\right)
\]
and similarly for $D\left(x;\frac{h}{k}\right)$ and $D\left(x_1;\frac{h}{k}\right)$. Hence
\begin{multline*}
D(x_1,x;\alpha)=e(\eta x)\left(k^{-1}(\log x+2\gamma-1-2\log k)x+E\left(0,\frac{h}{k}\right)+\Delta\left(x,\frac{h}{k}\right)\right)\\-e\left(\eta x_1\right)\left(k^{-1}(\log x_1+2\gamma-1-2\log k)x_1+E\left(0,\frac{h}{k}\right)+\Delta\left(x_1,\frac{h}{k}\right)\right)\\-2\pi i \eta\int_{x_1}^x \left(k^{-1}(\log t+2\gamma-1-2\log k)t+E\left(0,\frac{h}{k}\right)+\Delta\left(t,\frac{h}{k}\right)\right)e(\eta t)\ud t
\end{multline*}
Notice that
\[
e(\eta x)\Delta\left(x;\frac{h}{k}\right)-e(\eta x_1)\Delta\left(x_1;\alpha\right)-2\pi i \eta\int_{x_1}^x \Delta\left(t;\frac{h}{k}\right)e(\eta t)\ud t
\]
can be treated just like in the proof of Lemma 4 in \cite{jutila:ramanujan} . The same technique yields
\[
e(\eta x)\Delta\left(x;\frac{h}{k}\right)-e(\eta x_1)\Delta\left(x_1;\alpha\right)-2\pi i \eta\int_{x_1}^x \Delta\left(t;\frac{h}{k}\right)e(\eta t)\ud t\ll M^{1/2}\left(k^2\eta^2M\right)^{-a}
\]
for some positive constant $a$. Therefore, it suffices to work with the main term and the $E\left(0,\frac{h}{k}\right)$ term. First of all,
\begin{multline*}
-2\pi i \eta\int_{x_1}^x \left(k^{-1}t(\log t+2\gamma-1-2\log k)+E\left(0,\frac{h}{k}\right)\right)e(t\eta)\ud t\\ -\left[k^{-1}t\left(\log t+2\gamma-1-2\log k\right)e(t\eta)+E\left(0,\frac{h}{k}\right)e(t\eta)\right]_{x_1}^x+\int_{x_1}^x k^{-1}\left(\log t+1\right)e(t\eta)\ud t\\=-k^{-1}x(\log x+2\gamma-1-2\log k)e(x\eta)+k^{-1}x_1(\log x_1+2\gamma-1-2\log k)e(x_1\eta)\\-E\left(0,\frac{h}{k}\right)(e(x\eta)-e(x_1\eta)) + \int_{x_1}^x k^{-1}\left(\log t+2\gamma-2\log k\right)e(t\eta)\ud t
\end{multline*}
The first terms cancel out the main terms and the Esthermann zeta function term in (\ref{alku}), and therefore, it only remains to consider the integral
\[
\int_{x_1}^x k^{-1}\left(\log t+2\gamma-2\log k\right)e(t\eta)\ud t.
\]
Now use the first derivative test to obtain the estimate
\[
\int_{x_1}^x k^{-1}\left(\log t+2\gamma-2\log k\right)e(t\eta)\ud t\ll \frac{\log M}{k\eta} \ll M^{1/2}\log M\left(Mk^2\eta^2\right)^{-1/2}
\]
which proves the theorem.
\end{proof}

\section{Proof of the approximate functional equation}
Let us start with considering smoothed sums. Let $w$ be a smooth weight function which satisfies the condition
\[
\int_{M-JU+u}^{M+\Delta+JU-u}w(x)h(x)dx=U^{-J}\int_0^U\ud u_1 \cdots \int_0^U \ud u_J \int_M^{M+\Delta}h(x)dx
\]
for any test function $h(x)$ with $u=u_1+u_2+\cdots +u_J$. Recall $U=\eta^{-1/2}M^{1/2}\left(k^2\eta^2M\right)^{d}$ with $d$ a very small fixed positive number. Also, $J$ is a suitable large integer which will be defined later. The value of $J$ will only depend on $\varepsilon$, and therefore, while some constants in some estimates will depend on the weight function, those constant can be chosen in such a way that they only depend on $\varepsilon$.  Notice that since $k^2\eta^2M\gg 1$,
we have
\begin{multline*}
U=\eta^{-1/2}M^{1/2}\left(k^2\eta^2M\right)^{d}=\left(k\eta M^{1/2}\right)^{-1/2}k^{1/2}M^{1/4}M^{1/2}\left(k^2\eta^2M\right)^{d}\\=k^{1/2}M^{3/4}\left(k^2\eta^2M\right)^{d-1/4} \ll k^{1/2}M^{3/4}\ll M^{7/8}.
\end{multline*}

 Now
\begin{multline*}
\left|\sum_{M-JU\leq n\leq M+\Delta+JU} d(n)e(n\alpha)w(n)-\sum_{M\leq n\leq M+\Delta}d(n)e(n\alpha) \right|\\ =\left|\sum_{M-JU\leq n\leq M} d(n)e(n\alpha)w(n)+\sum_{M+\Delta \leq n\leq M+\Delta+JU}d(n)e(n\alpha)w(n) \right|\\ \leq \left|\sum_{M-JU\leq n\leq M} d(n)e(n\alpha)w(n)\right|+\left|\sum_{M+\Delta \leq n\leq M+\Delta+JU}d(n)e(n\alpha)w(n) \right|.
\end{multline*}
Theorem \ref{painoton} gives the estimate $\ll M^{1/2}(k^2\eta^2M)^{-a}$ for these sums.

Therefore, we can use the smoothed sum instead of the non-smoothed one.
Let us now use a Voronoi type summation formula (\ref{voronoi})
\begin{multline*}
\sum_{M\leq n\leq M+\Delta}d(n)e(n\alpha)w(n)=k^{-1}\int_M^{M+\Delta}w(x)e(\eta x)\left(\log x+2\gamma-2\log k\right)\ud x+\\ k^{-1}\sum_{n=1}^{\infty}d(n)\int_M^{M+\Delta}\left(-2\pi e\left(\frac{-n\bar{h}}{k}\right)Y_0\left(4\pi\frac{\sqrt{nx}}{k}\right)+4e\left(\frac{n\bar{h}}{k}\right)K_0\left(4\pi\frac{\sqrt{nx}}{k}\right)\right)w(x)e(\eta x)\ud x
\end{multline*}

The first term, and the term containing the K-Bessel function have already been treated in the Lemmas \ref{paa} and \ref{K-termi} (in Lemma \ref{K-termi}, there is no weight function, but that does not matter: we get rid of the weight function in this case just by taking absolute values). Therefore, it is sufficient to concentrate on the term containing the Y-Bessel function. Let us write the Y-Bessel function as a sum of exponential terms (see e.g. \cite{lebedev} (5.11.7)):
\begin{equation}\label{voronoitu}
Y_0(z)=\frac{-i}{\left(2\pi z\right)^{1/2}}\left(e^{z-\frac{1}{4}\pi}-e^{-z+\frac{1}{4}\pi}\right)+\left(\frac{1}{2\pi z}\right)^{3/2}\left(e^{z-\frac{1}{4}\pi}+e^{-z+\frac{1}{4}\pi}\right)+O\left(z^{-5/2}\right)
\end{equation}
Substitute now $z=\frac{4\pi\sqrt{nx}}{k}$. The sum containing the integration over the last term (the error term) is very easy to treat:
\[
k^{-1}\sum_{n=1}^{\infty}d(n)\int_M^{M+\Delta}\left(\frac{\sqrt{nx}}{k}\right)^{-5/4}\ud x=\Delta k^{1/4}M^{-5/4}
\] We will now state some lemmas which are proved using partial integration, and which will be used later to handle the terms arising from the Voronoi-type transformation and the use of the asymptotic expansion of a Bessel function.

The first term will yield the main contribution, and it is the most demanding to treat. The rest of the terms can be treated similarly, and they are much more straightforward. Therefore, we only briefly collect here the methods and contributions.

The part containing $\left(\frac{1}{2\pi z}\right)^{3/2}\left(e^{z-\frac{1}{4}\pi}+e^{-z+\frac{1}{4}\pi}\right)$ can be estimated using Lemma \ref{jutilamotohashi}, and the second derivative test (\cite{titchmarsh}, Theorem 5.9), and the contribution will be $\ll 1+(k^2\eta^2M)^{\varepsilon}$.

The sum containing the integration over the term $\frac{-i}{\left(2\pi z\right)^{1/2}}e^{z-\frac{1}{4}\pi}$ can also be treated by Lemma \ref{jutilamotohashi}, and it gives the contribution $\gg 1$. We may thus turn to the term yielding the main term. Substituting $z=\frac{4\pi\sqrt{nx}}{k}$ and plugging in the first term in the asymptotic expansion in the place of the Y-Bessel function yields the expression
\begin{multline*}
 -k^{-1}\sum_{n=1}^{\infty}d(n)\int_M^{M+\Delta}2\pi e\left(\frac{-n\bar{h}}{k}\right)\cdot \frac{i}{\left(2\pi 4\pi\frac{\sqrt{nx}}{k}\right)^{1/2}}\cdot e\left(-2\frac{\sqrt{nx}}{k}+\frac{1}{8}\right)w(x)e(\eta x)\ud x\\= -\frac{i}{\sqrt{2k}}\sum_{n=1}^{\infty}d(n)e\left(\frac{-n\bar{h}}{k}\right)n^{-1/4}\int_{M-JU}^{M+\Delta+JU}x^{-1/4}e\left(-\frac{2\sqrt{nx}}{k}+\eta x+\frac{1}{8}\right)w(x)\ud x
\end{multline*}

We first want to show that if $c_3$ is a large constant, then
\[
k^{-1/2}\sum_{n\geq c_3k^2\eta^2M}
d(n)e\left(-n\frac{\bar{h}}{k}\right)n^{-1/4}\int_{M-JU}^{M+\Delta+JU}w(x)x^{-1/4}e\left(\eta
x -2\frac{\sqrt{nx}}{k}\right)dx \ll 1.
\]

To prove this bound, we first use Lemma \ref{jutilamotohashi} with $A(x)=w(x)x^{-1/4}$,
$B(x)=\eta x-2\frac{\sqrt{nx}}{k}$, $A_0=M_{-1}^{-1/4}$, $A_1=U$,
$b-a=\Delta+2JU$, $B_1=\frac{\sqrt{n}}{k\sqrt{M}}$ and
$\varrho=M_{-1}$. This gives
\[
\int_{M-JU}^{M+\Delta+JU}w(x)x^{-1/4}e\left(\eta
x-2\frac{\sqrt{nx}}{k}\right)dx\ll
M^{-1/4}U^{-P}k^PM^{P/2}n^{-P/2}(U+\Delta)
\]
for any $P\leq J$. Substituting this estimate, the left-hand side
is dominated by
\begin{multline*}
\ll k^{-1/2}\sum_{n\geq
ck^2\eta^2M}n^{\varepsilon-1/4-P/2}M^{-1/4+P/2}U^{-P}k^{P}\left(U+\Delta\right) \\ \ll \left(U+\Delta\right)k^{-1/2}M^{-1/4}\left(k^2\eta^2M\right)^{\varepsilon-1/4}\left(\frac{\sqrt{M}k}{\sqrt{k^2\eta^2M}|\eta|^{-1/2}M^{1/2}\left(k^2\eta^2M\right)^d}\right)^P\\=\left(U+\Delta\right)k^{-1/2}M^{-1/4}\left(k^2\eta^2M\right)^{\varepsilon-1/4}\left(\left(k^2\eta^2M\right)^d\sqrt{M\eta}\right)^{-P}\ll 1,
\end{multline*}
when $P$ is sufficiently large.

\[
-\frac{i}{\sqrt{2k}}\sum_{n\leq c_3k^2\eta^2M}d(n)e\left(\frac{-n\bar{h}}{k}\right)n^{-1/4}\int_{M-JU}^{M+\Delta+JU}x^{-1/4}e\left(-\frac{2\sqrt{nx}}{k}+\eta x+\frac{1}{8}\right)w(x)\ud x
\]
Let us use the saddle point lemma to estimate the integral. The main term in the saddle point lemma only occurs if there is a saddle point term on the interval of integration. However, even if there is no saddle point, one can use the lemma. Let us now start with the case when there is a saddle point on the interval. Write $f(x)=-2\frac{\sqrt{nx}}{k}$. The saddle point is the point where
\[
f'(x)+\eta=-\frac{\sqrt{n}}{\sqrt{x}k}+\eta=0,
\]
that is, when $x_0=\frac{n}{k^2\eta^2}$. This point is on the interval when
\[
M-JU\leq \frac{n}{k^2\eta^2}\leq M+\Delta+JU \Leftrightarrow k^2\eta^2(M-JU)\leq n\leq k^2\eta^2(M+\Delta+JU).
\]
We will first calculate the main term in the saddle point lemma. Notice that
\[
 f(x_0)+\eta x_0=-2\frac{\sqrt{nx_0}}{k}+\eta x_0=-\frac{n}{\eta k^2}
\]
and
\[
 f''(x_0)=\frac{\sqrt{n}}{2x_0^{3/2}k}=\frac{\eta^3k^2}{2n}
\]

The saddle point lemma yields
\begin{multline*}
\int_{M-JU}^{M+\Delta+JU}x^{-1/4}e\left(-\frac{2\sqrt{nx}}{k}+\eta x+\frac{1}{8}\right)w(x)\ud x=\tilde{w}(n)\left(\frac{n}{k^2\eta^2}\right)^{-1/4}\left(\frac{\eta^3k^2}{2n}\right)^{-1/2}e\left(\frac{1}{8}-\frac{n}{\eta k^2}\right)\\=\tilde{w}(n)\sqrt{2}e\left(\frac{1}{8}-\frac{n}{\eta k^2}\right)\frac{n^{1/4}}{\eta\sqrt{k}}+O\left(\frac{k^{3/2}}{n^{3/4}}+\delta(n)\frac{M^{1/4}k}{\sqrt{n}}\right)\\+O\left(M^{-1/4}U^{-J}\sum_{j=1}^J\left(\left|\eta-\frac{\sqrt{n}}{\sqrt{M-jU}k}\right|+\frac{n^{1/4}}{\sqrt{k}M^{3/4}}\right)^{-J-1}\right)\\ +O\left(M^{-1/4}U^{-J}\sum_{j=1}^J\left(\left|\eta-\frac{\sqrt{n}}{\sqrt{M+\Delta+jU}k}\right|+\frac{n^{1/4}}{\sqrt{k}M^{3/4}}\right)^{-J-1}\right),
\end{multline*}
where
\[
\left\{\begin{array}{ll} \tilde{w}(n)=0\ \textrm{and } \delta(n)=0& \textrm{if $n\leq k^2\eta^2(M-JU)$ or $n\geq k^2\eta^2(M+\Delta+JU)$}\\
\tilde{w}(n)=1\ \textrm{and } \Delta(n)=1 & \textrm{if $k^2\eta^2M\leq n\leq k^2\eta^2(M+\Delta)$}\\
\tilde{w}(n)\ll 1,\ \tilde{w}'(n)\ll {k^2\eta^2U}-1\ \textrm{and}\delta(n) & \textrm{otherwise}\end{array}\right. 
\]
We may now calculate the sum over the main terms:
\begin{multline*}
 -\frac{i}{\sqrt{2k}}\sum_{k^2\eta^2(M-JU)\leq n\leq k^2\eta^2(M+\Delta+JU)}\tilde{w}(n)d(n)e\left(\frac{-n\bar{h}}{k}\right)n^{-1/4}\cdot \sqrt{2}e\left(\frac{1}{8}-\frac{n}{\eta k^2}\right)\frac{n^{1/4}}{\eta\sqrt{k}}\\=\frac{1}{k\eta}\sum_{k^2\eta^2(M-JU)\leq n\leq k^2\eta^2(M+\Delta+JU)}\tilde{w}(n)d(n)e\left(\frac{-n\bar{h}}{k}-\frac{n}{\eta k^2}\right)
\end{multline*}
When $k^2\eta^2M\ll n\ll k^2\eta^2(M+\Delta)$, ie. when $\tilde{w}(n)=1$, this yields the main term in the approximate functional equation. When $n$ is not on this interval, these sums contribute to error terms. Estimating them is rather simple. Since the length of the sum can be estimated to be
\[
k^2\eta^2U=k^2\eta^2(k^2\eta^2M)^d\eta^{-1/2}M^{1/2}= (k^2\eta^2M)^{d+1}(\eta M)^{-1/2}\ll (k^2\eta^2M)^{1/2+d}\ll (k^2\eta^2M)^{5/8},
\]
and therefore, we may use Theorem \ref{painotpois} and partial integration to estimate the sum:
\begin{multline*}
\frac{1}{k\eta}\sum_{k^2\eta^2(M-JU)\leq n\leq k^2\eta^2(M)}\tilde{w}(n)d(n)e\left(\frac{-n\bar{h}}{k}-\frac{n}{\eta k^2}\right)\ll \frac{1}{k\eta}\left(k^2\eta^2M\right)^{1/3+\varepsilon}(k^2\eta^2U)^{1/6}\\ \ll M^{1/2}(k^2\eta^2M)^{\varepsilon+d/6-1/12}.
\end{multline*}
The sum on the interval $[k^2\eta^2(M+\Delta),k^2\eta^2(M+\Delta+JU)]$ can be estimated similarly. We still need to treat the error terms arising from the use of the saddle point lemma. The first error term is very easy to estimate:
\[
k^{-1/2}\sum_{1\leq n\leq k^2\eta^2(M+\delta+JU)}n^{\varepsilon-1/4}\left(\frac{k^{3/2}}{n^{3/4}}+\delta(n)\frac{M^{1/4}k}{\sqrt{n}}\right)\ll k\left(k^2\eta^2M\right)^{\varepsilon+d}
\]
The sums containing the two other error terms are estimated in Lemma 2.8 in \cite{ame:approximate} by chopping the sums into three pieces depending on which part of the error term is the dominating one. This yields a total error of $\sqrt{M}(k^2\eta^2M)^{\varepsilon-Jd}$. Since we may choose J as large as we wish to ensure that $\varepsilon-Jd<-a$, this proves the theorem.


\begin{thebibliography}{10}

\bibitem{ame:approximate}
A.-M. Ernvall-Hyt\"onen.
\newblock On the error term in the approximate functional equation for
  exponential sums related to cusp forms.
\newblock {\em International Journal of Number Theory}, 4(5), 2008.

\bibitem{e&k}
A.-M. Ernvall-Hyt{\"o}nen and K.~Karppinen.
\newblock On short exponential sums involving {F}ourier coefficients of
  holomorphic cusp forms.
\newblock {\em Int. Math. Res. Not. IMRN}, (10):Art. ID. rnn022, 44, 2008.

\bibitem{jutila:divisor}
M.~Jutila.
\newblock On exponential sums involving the divisor function.
\newblock {\em J. Reine Angew. Math.}, 355:173--190, 1985.

\bibitem{jutila:lectures}
M.~Jutila.
\newblock {\em Lectures on a {M}ethod in the {T}heory of {E}xponential {S}ums},
  volume~80 of {\em Tata Institute of Fundamental Research Lectures on
  Mathematics and Physics}.
\newblock Published for the Tata Institute of Fundamental Research, Bombay,
  1987.

\bibitem{jutila:ramanujan}
M.~Jutila.
\newblock On exponential sums involving the {R}amanujan function.
\newblock {\em Proc. Indian Acad. Sci. Math. Sci.}, 97(1-3):157--166 (1988),
  1987.

\bibitem{jutimoto:acta}
M.~Jutila and Y.~Motohashi.
\newblock Uniform bound for {H}ecke {$L$}-functions.
\newblock {\em Acta Math.}, 195:61--115, 2005.

\bibitem{karppi}
K.~Karppinen.
\newblock {\em Exponential {S}ums {R}elating to {C}usp {F}orms (Finnish)}.
\newblock University of Turku, 1998.
\newblock Licenciate thesis.

\bibitem{lebedev}
N.~N. Lebedev.
\newblock {\em Special {F}unctions and {T}heir {A}pplications}.
\newblock Dover Publications Inc., New York, 1972.
\newblock Revised edition, translated from the Russian and edited by Richard A.
  Silverman.

\bibitem{titchmarsh}
E.~C. Titchmarsh.
\newblock {\em The {T}heory of the {R}iemann {Z}eta-{F}unction}.
\newblock Oxford, at the Clarendon Press, 1951.

\bibitem{wilton:divisor}
J.~R. Wilton.
\newblock An approximate functional equation with applications to a problem of
  {D}iophantine approximation.
\newblock {\em J. Reine Angew. Math.}, 169:219--237, 1933.

\end{thebibliography}
\end{document}